\documentclass[11pt]{article}
\usepackage{graphics}
\usepackage{tikz}
\usepackage{amsfonts,mathrsfs, }
\usepackage{amsbsy}
\usepackage{amssymb}
\usepackage{amsmath,amsthm}
\usepackage[scr=rsfs,cal=boondox]{mathalfa}
\usepackage{verbatim}
\usepackage{amsfonts}
\usepackage{pst-all}
\usepackage{pstricks}
\usepackage[T1]{fontenc}
\usepackage{multicol}
\usepackage{color}
\usepackage[autostyle]{csquotes}
\usepackage{pgf}
\usepackage[symbol]{footmisc}
\usepackage{mathtools}
\usepackage{hyperref}
\usepackage{wasysym}      
    \usepackage{amsmath} 
    \usepackage{arydshln}
\usepackage{scalerel,amssymb}

\hypersetup{
    colorlinks=true,
    linkcolor=blue,
    filecolor=blue,      
    urlcolor=blue,
    citecolor=blue,
    pdftitle={Overleaf Example},
    pdfpagemode=FullScreen,
    }
\DeclarePairedDelimiter\ceil{\lceil}{\rceil}
\DeclarePairedDelimiter\floor{\lfloor}{\rfloor} 

\newtheorem{thm}{Theorem}[section]

\newtheorem{prop}{Proposition}[section]
\newtheorem{cor}[thm]{Corollary}
\newtheorem{defn}[thm]{Definition}

\newtheorem{lem}[thm]{Lemma}

\newtheorem{rem}[prop]{Remark}

\newcommand{\B}{\mathcal{B}}

\newcommand{\G}{\mathcal{G}}

\newcommand{\bH}{\mathbb{H}}

\newcommand{\Z}{\mathbb{Z}}

\RequirePackage{pgfkeys}
\pgfkeys{/ytableau/options/.is family}

\pgfkeys{/ytableau/options, boxsize/.value required,  boxsize/.code = {%
\pgfkeysalso{nosmalltableaux}%
 \compare@YT{#1}{normal}%
 \ifeq@YT
 \xdef\macro@boxdim@YT{\expandonce@YT\boxdim@normal@YT}%
 \else
 \xdef\macro@boxdim@YT{#1}%
 \fi
 } }
 \pgfkeys{/ytableau/options, textmode/.value forbidden,
 textmode/.code = \set@textmode@YT,
mathmode/.value forbidden,
mathmode/.code = \set@mathmode@YT,
 }
 
\def\set@mathmode@YT{
 \gdef\skipin@YT{$}
 \gdef\skipout@YT{$}
 \def\smallfont@YT{\scriptstyle} }
 \set@mathmode@YT

\makeatletter
\let\@fnsymbol\@arabic
\makeatother
\date{}

\title{On the expansion constant and distance constrained  colourings of hypergraphs}
\author {Annayat Ali$^{1}$ and Rameez Raja$^{2}$\footnote{$^{, 2}$Department of Mathematics, National Institute of Technology Srinagar, Jammu and Kashmir, India. Email: annayat\_05phd20@nitsri.net, rameeznaqash@nitsri.ac.in, 
Corresponding author:$^2$}}

\begin{document}
\maketitle
\begin{abstract} 
For any two non-negative integers $h$ and $k$, $h > k$, an \textit{$L(h,k)$-colouring} of a graph $G$ is a colouring of vertices such that adjacent vertices admit colours that at least  differ by  $h$ and vertices that are two distances apart admit colours that at least differ by $k$. The smallest positive integer $\delta$ such that $G$ permits an $L(h,k)$-colouring with maximum colour $\delta$ is known as the \textit{$L(h,k)$-chromatic number} ($L(h,k)$-colouring number) denoted by $\lambda_{h,k}(G)$.  In this paper, we discuss some interesting invariants in hypergraphs. In fact, we study the relation between the spectral gap and $L(2,1)$-chromatic number of hypergraphs. We derive some inequalities which relates $L(2,1)$-chromatic number of a $k$-regular simple graph to its spectral gap and expansion constant. The upper bound of $L(h,k)$-chromatic number in terms of various hypergraph invariants such as \textit{strong chromatic number}, \textit{strong independent number} and maximum degree is obtained. We   determine the sharp upper bound for $L(2,1)$-chromatic number of hypertrees in terms of its maximum degree. Finally, we conclude this paper with a discussion on $L(2,1)$-colouring in cartesian product of some classes of hypergraphs.
\end{abstract}
\textit{2020 AMS Classification Code}: 05C15, 05C50, 05C65, 05C78.
   
\section{Introduction}
\label{s1}
A \textit{hypergraph} $\bH=(V,E)$ on a finite vertex set $V$ is a collection of non-empty subsets  $E=\{e_i:e_i \subset V, ~i \in I\}$, where $I$ is an indexing set. $\bH$ is said to be \textit{simple} if no edge contains another and every edge has a cardinality greater than one. If the intersection of any two edges contains at most one element, then $\bH$ is said to be \textit{linear}. The maximum cardinality of a edge in $\bH$  is called its \textit{rank} denoted by $r(\bH)$ and the minimum cardinality of a edge  is called its \textit{corank} denoted by  $cr(\bH)$. If $r(\bH) = cr(\bH)=r$, then $\bH$ is said to be \textit{$r$-uniform}. The star $\bH(v)$ rooted at $v \in V$  is the collection of all edges $e$ containing $v$.  The degree of a vertex $v \in V$, denoted by $deg(v)$, is equal to the number of edges containing $v$ except for a loop $\{v\}$ which contributes $2$ to the degree of $v$. All the hypergraphs considered in this article are simple. 

For any two non-negative integers $h$ and $k$, $h > k$, an  $L(h,k)$-colouring of a hypergraph $\bH$ is a function $f:V \longrightarrow \mathbb{Z}_{\geq 0}$ such that $|f(v)-f(u)|\geq h$, whenever  $v,u \in e$ and $|f(v)-f(u)| \geq  k$, whenever  there exists $e_1, e_2 \in E$  with $v \in e_1$, $u \in e_2$ and $(e_1\cap e_2)\setminus\{v,u\} \neq \phi$. The difference between the maximun and the minimum value of $f$ is called its \textit{span}. The minimum span over all  $L(h,k)$-colourings of a hypergraph  is known as the \textit{$L(h,k)$-chromatic number} of $\bH$, denoted by $\lambda_{h,k}(\bH)$, and the corresponding  colouring $f$ is called a minimal \textit{$L(h,k)$-colouring}. We refer to $L(2,1)$-chromatic number of $\bH$ as its \textit{$\lambda$-chromatic number} denoted by $\lambda(\bH)$.

The \textit{expansion constant} or \textit{isoperimetric constant} of a simple connected graph $G=(V, E)$ denoted by $h(G)$, is defined as, $h(G)= \inf\{\frac{|\partial S|}{\min\{|S|,|V-S|\}}: S \subseteq V,  ~0 < |S| <  \infty\}$, where $\partial S$ is the set of edges in $G$ connecting $S$ to its complement in $G$ and it is called as the \textit{boundary set} of $S$.  

The motivation of study of $L(h,k)$-colourings in hypergraphs is more general than graphs. We allow radio transmitters at a region to correspond to the vertices of a hypergraph with edges consisting of transmitters that are close to each other. Defining a $L(h,k)$-colouring on the hypergraph is equivalent to assigning channel frequencies to radio transmitters in such a manner that neighbouring radio channels do not interfere with each other. On the other hand,  expansion constant is one of the very interesting measures of the connectivity of a graph $G$. If we consider $G$ to be a network that transmits information, then $h(G)$ measures the effectiveness of $G$ as a network. In fact, if $h(G)$ is large, then $G$ is considered to be an efficient network.

The research article is divided into four sections. In section \eqref{s2}, we study invariants such as expansion constant, spectral gap and $\lambda$-chromatic number in graphs. Further, we discuss the relationship between spectral gap and $\lambda$-chromatic number of $\bH$. Section \eqref{s3} consists of three subsections, in the first subsection, we present some upper bounds for $L(h,k)$-chromatic number of $\bH$ in terms of hypergraph invariants. In the second subsection, we relate the parameter $\lambda(\bH)$ to $\chi(\mathcal{G})$, where $\chi(\mathcal{G})$ denotes the usual chromatic number of the line graph associated with $\bH$. We prove that $\lambda_{h, k}([\bH]_s) = \lambda_{h, k}(\bH)$, where $[\bH]_s$ denotes the \textit{s-section} of $\bH$. The third subsection is devoted to study the $\lambda$-chromatic number of hypertrees. Finally, in section \eqref{s4}, we determine the $\lambda$-chromatic number of the cartesian product of some classes of hypergraphs.

\section{Expansion constant, spectral gap and $\lambda$-chromatic number}
\label{s2}
 The adjacency matrix of $\bH$ with $n$ vertices is a $n \times n$ matrix with $(i, j)^{th}$ entry $1$ if $v_i$  is adjacent to $v_j$, otherwise zero. The \textit{spectrum} of $\bH$ is the set of $n$ eigenvalues $\{\mu_0 \geq \mu_1 \geq  \cdots \geq  \mu_{n-1}\}$ of the adjacency matrix of $\bH$. Its \textit{spectral gap} is the difference $ \mu_0 - \mu_1$  between the largest and second largest eigen values. The expansion constant for a finite graph $G=(V,E)$ of order $n$ can be also rephrased as, $h(G)= \min\{\frac{|\partial F|}{|F|}: F \subset V, 0 < |F| \leq \frac{n}{2}\}$. It can be verified that if $G$ is a connected $k$-regular graph with $n$ vertices, then
$\mu_0 = k  >  \mu_1 \geq \cdots  \geq \mu_{n-1} \geq -k$.

The following result due to Alon and Milman \cite{AM} and Dodziuk \cite{D} estimates the expansion constant of a finite connected $k$-regular graph in terms of  the spectral gap.
\begin{thm}
\label{ec}
let $G$ be a finite connected  $k$-regular graph. Then the following hold, 
$$\frac{k-\mu_1}{2} \leq  h(G) \leq \sqrt{2k(k-\mu_1)}~.$$
\end{thm}

The succeeding result investigates the relation of $\lambda$-chromatic number of a finite connected $k$-regular graph $G$ to its spectral gap and expansion constant. In fact, provide an upper bound for $h(G)$ in terms of $\lambda(G)$.
  \begin{thm}
  \label{lc}
  Let $G$ be a finite connected $k$ regular graph on $n$ vertices. Then  the following hold,
  $$h(G)  <k\lambda(G)\sqrt{\frac{n}{\floor{\frac{n}{2}}}}~.$$
  \end{thm}
  \begin{proof}
 Consider two real linear spaces $l^2(V)=\{g:V\longrightarrow \mathbb{R}\} \cong \mathbb{R}^n$  and $l^2(E)=\{f:V\longrightarrow \mathbb{R}\} \cong \mathbb{R}^m$, where $|V| = n$ and $|E| = m$. Define the inner product$<,>$ on $l^2(V)$ by $<g,h> = \sum\limits_{x \in V}g(x)h(x)$. It can be defined analogously on $l^2(E)$. Choose a function $g \in l^2(V) $ such that  $|supp(g(x))| \leq \frac{n}{2}$, where $supp(g(x)) =\{ x \in V | g(x) \neq 0\}$. We orient  edges of $G$ arbitrarily and associate $e^-$  as an origin for any $e \in E$ and $e^+$ its extremity. This orientation allows us to define a function $d:l^2(V) \longrightarrow l^2(E),$ such that
$dg(e) = g(e^+) - g(e^-)$,  where $g \in l^2(V)$. 
  
We set $\B_g = \sum\limits_{e \in E} |g(e^+)^{2} - g(e^-)^{2}|$ and let  $g(V)=\{ 0=\beta_0 < \beta_1 < \beta_2 < \cdots < \beta_r =\lambda(G)\}$, $L_i  =\{ x \in V ~|~g(x) \geq \beta_i\}$. Then, $L_0 =V$, $\partial L_0 = \phi$ and $L_r \subsetneq  L_{r-1} \subsetneq \cdots  \subsetneq L_0 =V $. Note that $\partial L_i$ consists of edges that connect a vertex $x \in L_i$ to a vertex $y$ with $g(y) < \beta_i$. Now, we have the following claims.\\
{\bf Claim-1:} $\B_g = \sum\limits _{i=1}^r|\partial L_i|(\beta^2_i - \beta^2_{i-1})$.\\
To prove the claim, we consider the set $E_g $ which consists of all edges $e \in E$ with $g(e^+) \neq g(e^-)$. Clearly, $\B_g = \sum\limits _{ e \in E_g} |g(e^+)^{2} - g(e^-)^{2}|$. Any edge  $e \in E_g$ connects a vertex $x$  with $g(x)= \beta_{i(e)}$ to a vertex $y$ with $g(y)= \beta_{j(e)}$, where $i(e) >j(e)$. Therefore, 
  \begin{eqnarray*}
  \B_g &= & \sum\limits _{ e \in E_g}(\beta^2_{i(e)}-  \beta^2_{j(e)})\\
  &=& \sum\limits _{ e \in E_g}(\beta^2_{i(e)}- \beta^2_{i(e)-1} + \beta^2_{i(e)-1} - \cdots -\beta^2_{j(e)+1}+\beta^2_{j(e)+1}-\beta^2_{j(e)} )\\
  &=& \sum\limits _{ e \in E_g}\sum\limits_{t=j(e)+1}^{i(e)}(\beta^2_{t}-  \beta^2_{t-1}).
  \end{eqnarray*}
An edge that connects two vertices $x$  and $y$ with $g(x)= \beta_{i(e)}$ and  $g(y)= \beta_{j(e)}$ passes through each level $L_t$, $j(e) <t <i(e)$. This corresponds to expanding the term $(\beta^2_{i(e)}-  \beta^2_{j(e)})$ in the expression of $ \B_g$ by introducing the zero difference $(\beta^2_{t}-  \beta^2_{t})$  for each level curve $L_t$ that is passed by an edge $e$. This implies that the zero term $(\beta^2_{t}-  \beta^2_{t})$ appears in $\B_g$ for every edge $e$ connecting some vertex  $x$  with $g(x)= \beta_{i}$, $i \geq t$ to some vertex $y$ with $g(y)= \beta_{j}$, $j  < l$. In other words, the difference appears for every edge $e \in \partial L_i$. Therefore the Claim-1 is established, that is,
  \begin{equation}
  \label{1}
  \B_g = \sum\limits _{i=1}^r|\partial L_i|(\beta^2_i - \beta^2_{i-1}).
  \end{equation}
  {\bf Claim-2:} $\B_g \leq \sqrt{2k}\|dg\|\|g\|$.\\
By Cauchy-Schwarz inequality and the fact that $(a+b)^2\leq 2(a^2+b^2)$, 
   \begin{eqnarray*}
    \B_g &= & \sum\limits_{e \in E} (g(e^+) + g(e^-))(g(e^+) - g(e^-))\\
       &\leq& \bigg[ \big(g(e^+) + g(e^-)\big)^2 \bigg]^{\frac{1}{2}}\bigg[\big(g(e^+) - g(e^-)\big)^2\bigg]^{\frac{1}{2}}\\
       &\leq & \sqrt{2}\bigg[\big(g(e^+) + g(e^-)\big)^2\bigg]^{\frac{1}{2}}\|dg\| \\
       &=& \sqrt{2k}\|g\|\|dg\|.
    \end{eqnarray*}
 This proves Claim-2 and therefore,
    \begin{equation}
    \label{2}
    \B_g \leq \sqrt{2k}\|dg\|\|g\|.
    \end{equation}
\\
    {\bf Claim-3:} $\B_g \geq h(G) \|g\|^2$.\\
Since $|supp(g(x)| \leq \frac{n}{2}$ , therefore, $|L_i|  \leq  \frac{n}{2}$, where $i = 1, 2, \cdots, r$. By the definition of $h(G)$, $h(G)|L_i| \leq |\partial L_i|$. Therefore by \eqref{1},
  \begin{eqnarray*}
  \B_g &\geq& h(G) \sum\limits _{i=1}^r| L_i|(\beta^2_i - \beta^2_{i-1})\\
  &=&h(G)\bigg(|L_r|\beta^2_r +(|L_{r-1}|-|L_r|)\beta^2_{r-1}+ \cdots +(|L_1|-|L_2|)\beta^2_1\bigg)\\
  &=&h(G)\bigg( |L_r|\beta^2_r +  \sum\limits _{i=1}^{r-1}(|L_i-L_{i+1}|)\beta^2_i\bigg).
   \end{eqnarray*}
  Due to the fact that the set $L_i-L_{i+1}$ contains exactly those vertices for which $g$ takes value $\beta_i$, therefore it follows,
   \begin{equation}
   \label{3}
   \B_g \geq h(G) \|g\|^2.
   \end{equation}
   Thus the Claim-3 follows.
 Furthermore,  By \eqref{2} and \eqref{3}, 
    \begin{equation}
    \label{4}
   h(G) \|g\| \leq   \sqrt{2k}\|dg\|.
    \end{equation}
    Since $\lambda(G) \geq 2$, therefore,
 $\|g\| > \sqrt{\floor{\frac{n}{2}}}$ and  $\|dg\| < \sqrt{\frac{kn}{2}}\lambda(G)$. This implies, \begin{equation}
 \label{eq5}
  h(G) \|g\|> h(G) \sqrt{\floor{\frac{n}{2}}}. \end{equation}
  Using \eqref{eq5} in \eqref{4},
        \begin{eqnarray*}
        h(G)  &<&k\lambda(G)\sqrt{\frac{n}{\floor{\frac{n}{2}}}}~.
         \end{eqnarray*}
                   \end{proof} 
From  Theorems \eqref{ec} and \eqref{lc}, we have the following immediate consequence.
\begin{cor}
\label{cor}
 Let $G$ be a finite connected $k$-regular graph on $n$ vertices. Then the following hold, 
 $$k-\mu_1 < 2k\lambda(G)\sqrt{\frac{n}{\floor{\frac{n}{2}}}}~.$$
 \end{cor}
For $2\leq s\leq r$, the \textit{$s$-section} of $\bH$  denoted by $[\bH]_s$ is a hypergraph with vertex set  $V(\bH)$ with edges $E([\bH]_s)=\{e^\prime\}$ satisfying either $e^\prime \subseteq e \in E(\bH)$, $|e^\prime|= s$ or  $e^\prime =e$ if $|e|<s$. Observe that $[\bH]_s$ is also a simple hypergraph with  $r([\bH]_s) =s.$ Below, we relate $L(h,k)$-chromatic number of $\bH$  to the $L(h,k)$-chromatic number of its $2$-section.
 \begin{lem}\label{lem}
If $[\bH]_2$ is the $2$-section of  $\bH$, then $\lambda_{h,k}(\bH) = \lambda_{h,k}([\bH]_2)$.

\end{lem}

\begin{proof}
Let $f$ be a minimal $L(h,k)$-colouring of  $\bH$ . Since  $V(\bH)=V([\bH]_2)$, therefore $f$ is a vertex colouring of $[\bH]_2$. By the adjacency relation on $[\bH]_2$, a vertex $v_1$ is adjacent to vertex $v_2$ if and only if there exists $e \in E(\bH)$ containing both $v_1$ and $v_2$. Consequently, $|f(v_1)-f(v_2)| \geq j$, whenever $v_1$ is adjacent to $v_2$ in  $[\bH]_2$. Further, $d(v_1,v_2) =2$ in  $[\bH]_2$  if and only if there exists edges $e_1, e_2 \in E(\bH)$ such that $e_1 \cap e_2$ contains a vertex distinct from $v_1$ and $v_2$. Thus, $|f(v_1)-f(v_2)| \geq k$, whenever $d(v_1,v_2) =2$. This implies, $f$ is  well defined  $L(h,k)$-colouring of $[\bH]_2$ as well. To prove that $f$ is a minimal  $L(h,k)$-colouring of $[\bH]_2$, assume to contrary that there exists another  $L(h,k)$-colouring $g$ of   $[\bH]_2$ with $span(g) <span(f)$. It is easy to verify that $g$ is a well defined $L(h,k)$-colouring of $\bH$. Therefore, $span(g)<\lambda_{h,k}(\bH)$, a contradiction and the result follows.
\end{proof}
Note that the adjacency matrix of a hypergraph is the same as that of its $2$-section. Thus it follows that the spectrum of both the graphs are same and this allows us to deduce the following result from Corollary \eqref{cor} and Lemma \eqref{lem}.
\begin{cor}
Let $\bH$ be a connected $k$-regular, $r$-uniform hypergraph on $n$ vertices. Then the following hold, 
$$k-\mu_1 < 2k(r-1)\lambda(\bH)\sqrt{\frac{n}{\floor{\frac{n}{2}}}}~.$$
\end{cor}

\section{$L(h,k)$-chromatic number of hypergraphs}
\label{s3}
This section of the article is devoted to study in detail the $\lambda_{h, k}$-chromatic number of  hypergraphs. We obtain different bounds for $\lambda_{h, k}(\bH)$ in terms of hypergraph invariants such as \textit{strong chromatic number}, \textit{strong independence  number} and maximum degree. We further discuss the relationship between $\lambda(\bH)$ and $\chi(\mathcal{G})$ and prove that $\lambda_{h,k}([\bH]_s) = \lambda_{h,k}(\bH)$. We provide sharp bounds for the $\lambda$-chromatic number of hypertrees.
\subsection{Upperbound of $\lambda_{h,k}.$}
A \textit{strong k-coloring} of $\bH$ is a partition $\{ C_1, C_2, \cdots , C_k\} $ of $V$  such that $|C_i \cap e|\leq 1$ for all $e \in E$, $1 \leq i \leq k$. The smallest $k$ such that $\bH$ has a strong $k$-coloring is known as the \textit{strong chromatic number} of  $\bH$ denoted by $\chi^\prime(\bH)$.  We provide an upper bound for $\lambda_{h,k}(\bH)$ in terms of $\chi^\prime(\bH)$.
\begin{thm}
\label{thm2.1}
If $\bH$ is a hypergraph of order $n$ with $\chi^\prime(\bH) =k^\prime$, then the following hold, 
$$\lambda_{h,k}(\bH) \leq k(n-k^\prime) +(k^\prime -1)h.$$
\end{thm}
\begin{proof}
Since $\chi^\prime(\bH) =k^\prime$, we partition the vertex set $V$ of $\bH$ as $\big\{ V_1, V_2, \cdots , V_{k^\prime} \big\}$ such that each $V_i$, $1 \leq i \leq k^\prime$  is monochromatic. Thus by the definition, $|e \cap V_j| \leq 1$ for all $1 \leq i \leq k^\prime$ and $ e \in E(\bH)$. For each $i, ~ 1 \leq i \leq k^\prime$, denote the vertices in $V_i$  by $v_{j,i}$, where $1 \leq j \leq n_i = |V_i|$.  Define, $L(h,k)$-colouring $f: V(\bH) \longrightarrow \mathbb{Z}_{\geq 0}$ by, 
\begin{eqnarray*}
f(v_{j,1})&=& k(j-1), 1 \leq j \leq n_1. \\
\mbox{ For }  2 \leq i \le	 k^\prime  \mbox{ and } 1\leq j \leq n_j, \mbox{ define,}\\
f(v_{j,i})&=&k\sum\limits_{t=1}^{i-1}n_t + k(j-i) +h(i-1). \\
\mbox{ Therefore,} ~span(f)&=& k(n-k^\prime) +(k^\prime -1)h.
\end{eqnarray*}
It is clear that $f$ is well defined $L(h,k)$-colouring of $\bH$, since any vertex set $V_i$ contains at most one vertex from each edge $e \in E(\bH)$ and consequently,
$$\lambda_{h,k}(\bH) \leq k(n-k^\prime) +(k^\prime -1)h.$$

\end{proof}
\begin{rem} The inequality given in Theorem \eqref{thm2.1} is sharp for the complete $r$-uniform hypergraph $\bH_n^r$ on $n$ vertices, $2 \leq r \leq n$, since $\lambda_{h,k}(\bH_n^r) =h(n-1)$ and  $\chi^\prime(\bH_n^r) = n $.
\end{rem}
For a hypergraph  $\bH$, a set $W \subseteq V$ is called a \textit{strong stable set} if $|W\cap e| \leq 1$ for all $e \in E$. The maximum cardinality of a  strong stable set of $\bH$ is known as \textit{ strong independence number}, denoted by $\bar{\alpha}(\bH)$. We determine another upper bound for $\lambda_{h,k}(\bH)$ in terms of  $\bar{\alpha}(\bH)$.
\begin{thm}
Let $\bH$ be a hypergraph of order $n$. Then the following hold, 
$$\lambda_{h,k}(\bH) \leq nh +\bar{\alpha}(\bH)(k-h) -k.$$
\end{thm}
\begin{proof} 
Let $W$ be a strong stable set of $\bH$. Then, $|V\setminus W| = n- \bar{\alpha}(\bH)$. Denote the vertices in $W$ by $w_i$, $1 \leq i \leq \bar{\alpha}(\bH)$ and vertices in $V\setminus W$ by $v_j$, $1 \leq j \leq   n- \bar{\alpha}(\bH)$. Define $L(h,k)$-colouring $f: V(\bH) \longrightarrow \mathbb{Z}_{\geq 0}$ by, 
\begin{eqnarray*}
f(w_i)&=& k(i-1),\\
f(v_j)&=& k(\bar{\alpha}(\bH) -1) +jh.\\
\mbox{ Therefore,} ~span
(f) &=& nh +\bar{\alpha}(\bH)(k-h) -k.
\end{eqnarray*}
It can be verified that $f$ is a well defined $L(h,k)$-colouring of $\bH$, since  $W$ contains at most one vertex from each $e \in E(\bH)$. Thus,
$$\lambda_{h,k}(\bH) \leq nh +\bar{\alpha}(\bH)(k-h) -k.$$
\end{proof}

The following result generalizes a result established by chang and kuo \cite{CK} to hypergraphs, therefore providing an upper bound for $\lambda(\bH)$ in terms of its maximum degree $\Delta$ and rank $r$. By putting $r = 2$, we can obtain the result of chang and kuo proved in \cite{CK}.
\begin{thm}
If $\bH$ is a  hypergraph with $r(\bH)= r$ and maximum  degree $\Delta$, then $\lambda(\bH) \leq  (r-1)(\Delta +1)\Delta$.
\end{thm}
\begin{proof}
We arbitrarily order the vertices of $\bH$  and colour them. We start from the least possible integer. Any vertex $v \in V(\bH)$ is adjacent to at most  $(r-1)\Delta$  vertices  and there are at most  $\Delta(\Delta -1)$ edges   whose pairwise intersection with those edges containing $v$  contains a vertex other than $v$. To colour the vertex $v$, we need to avoid at most $(r-1)(\Delta -1)\Delta + 2(r-1)\Delta =  (r-1)(\Delta +1)\Delta $ colours. Therefore the  result follows.
\end{proof}

Gonclaves \cite{G} gives the upper bound of $\lambda_{h,1}$ in terms of maximum degree for simple graphs. 
\begin{thm}[\cite{G}] 
\label{h1}
For any positive integer $h \geq 2$ and for any graph $G$ with maximum degree $\Delta \geq 3$, $\lambda_{h,1}(G) \leq \Delta^2 +(h-1)\Delta-2$.
\end{thm}
Chang, G. J., et. al \cite{CJ} demonstrated the following noteworthy result on the ratio $\frac{\lambda_{h+1,1}(\bH)}{\lambda_{h,1}(\bH)}$ as $h$ approaches to infinity.
\begin{thm}[\cite{CJ}]
\label{ratio}
If $G$ is a graph with at least one edge, then $\lim\limits_{h \longrightarrow \infty}\frac{\lambda_{h+1,1}(G)}{\lambda_{h,1}(G)} = 1$.
\end{thm}
Suppose $\bH$ is a $r$-uniform hypergraph with  maximum degree $\Delta$. Then the maximum degree of a vertex in $[\bH]_2$ is $(r-1)\Delta$. So, $\Delta([\bH]_2) \leq (r-1)\Delta$. By Lemma \eqref{lem} and Theorem \eqref{h1},
$$\lambda_{h,1}(\bH) = \lambda_{h,1}([\bH]_2) \leq (r-1)\Delta\big(\Delta r +h - \Delta -1\big)-2.$$ 
Furthermore, by Theorem \eqref{ratio}, 
$$\lim_{h \longrightarrow \infty}\frac{\lambda_{h+1,1}(\bH)}{\lambda_{h,1}(\bH)} = 1$$

For a graph $G$ with diameter two, Griggs and Yehh \cite{GY} established a better upper bound of $\lambda(G)$ in terms of maximum degree of $G$. 
\begin{thm}[\cite{GY}]
\label{d2}
If $G$ is a graph with diameter two, then $\lambda(G) \leq \Delta^2$.
\end{thm} 
If $\bH$ is an $r$-uniform hypergraph of diameter two, then for any two edges $e_1, e_2 \in E(\bH)$,  we have $e_1 \cap e_2 \neq \phi$. This implies, $diam([\bH]_2) \leq 2 $ with maximum degree $(r-1)\Delta .$ Thus by Theorem \eqref{d2} and Lemma \eqref{lem}, $ \lambda(\bH ) = \lambda([\bH]_2)   \leq  \big((r-1)\Delta\big)^2$. 

\subsection{On graphs associated with hypergraphs}
Let  $\bH$ be a hypergraph  without isolated vertices. The \textit{line graph} of $\bH$, denoted by $\mathcal{G}$, is a simple graph whose vertex set $V(\mathcal{G})$ is edge set $E(\bH)$ and two vertices are adjacent in $\mathcal{G}$ if and only if  their corresponding edges in $\bH$ have a non empty intersection. If $\mathcal{G}$ is a line graph associated to $\bH$, then $\bH$ is not unique. Below, we present an example of two non isomorphic linear $3$-uniform hypergraphs with isomorphic line graph.
 \begin{figure}[h]
 \centering
 \includegraphics[scale=.30]{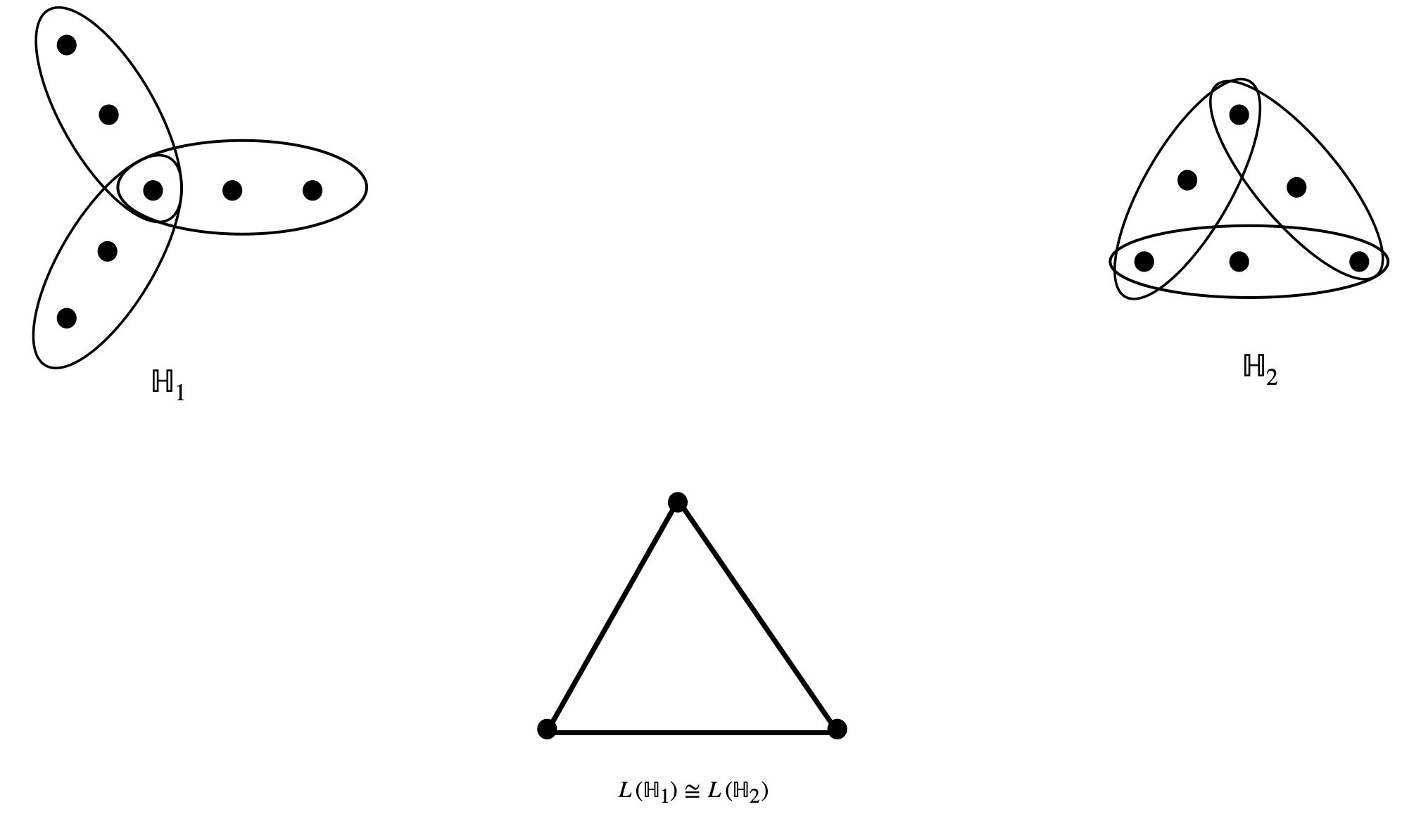}
     \caption{Two non-isomorphic hypergraphs with isomorphic Line graph.}
 \label{Fig:1}
 \end{figure}
We provide a relationship between the chromatic numbers $\chi(\G)$ and $\lambda(\bH)$. The result is significant in a sense that it gives the upper bound of $\lambda(\bH)$ for any hypergraph $\bH$ in the  collection of non-isomorphic hypergraphs with isomorphic line graphs.
 \begin{thm}
If $\G$ is a line graph  of some $r$-uniform linear hypergraph $\bH$  with $\chi (\G) = k$, then $\lambda(\bH) \leq 2kr-2$.
\end{thm}
\begin{proof}
If $k$ is the chromatic number of $\G$, then $\lambda_{1,0}(\G)=k-1$, since we can colour the vertices using $0$ as well. Let $f$ be a minimal $L(1,0)$-colouring of $\G$ with $f(V(\G)) =\big\{ 0, 1, \cdots , k-1\}.$ Define an equivalence relation $\sim$ on $V(\G)$ as, $v \sim u$  if and only if $f(v) = f(u)$. Consequently, we have the same equivalence relation on $E(\bH)$.
For $0 \leq j \leq k-1$, let  $v_j$ be the representative  of an equivalence class for which $f(v_j)= j$ and let $e_j \in E(\bH) $ be the edge corresponding to vertex $v_j \in V(\G)$. Consider the ordered set $ e_j = \big\{v^j_l: 1 \leq l \leq r\big\}$ and define $\bar{f}:V(\bH) \longrightarrow \Z_{\geq 0}$ by,
\begin{eqnarray*}
\bar{f}(v^0_l)&=& 2l-2,\\ 
\bar{f}(v^j_l)&=&2jr+2l-2, \mbox{ where } 0<j \leq k-1.
\end{eqnarray*}
Therefore, $span(\bar{f}) = 2kr-2$. The map $\bar{f}$ gives colors to all vertices of $\bH$ representing edges in each equivalence classes of $\sim$ on $E(\bH)$. However, if a vertex $v$ appears in more than one edge, then colour $v$ by the minimum color which $v$ inherits from the above colouring. Clearly, $\bar{f}$ is a well defined $L(2,1)$-colouring of $\bH$, since $f$ is well defined $L(1,0)$-colouring of $\G$. Therefore the result follows.
\end{proof}
Next, we present an extension of Lemma \eqref{lem}.
\begin{thm}
If $[\bH]_s$ is the $s$-section of  $\bH, 2<s\leq rank(\bH)$, then  the following hold,
$$\lambda_{h,k}(\bH) = \lambda_{h,k}([\bH]_s).$$
\end{thm}
\begin{proof}
Let $f$ be a minimal $L(h,k)$-colouring of $[\bH]_s$. We show that $f$ is a well defined  $L(h,k)$-colouring of $\bH$.  Since $V([\bH]_r) = V(\bH)$  and $v_1, v_2$ are adjacent in $[\bH]_r$ if and only if  $v_1, v_2 \in e$ for some  $e \in E(\bH)$. Therefore,
$|f(v_1) - f(v_2)|$. Let $v_1 \in e_1$ and $v_2 \in e_2 \in E(\bH)$. Suppose that there exists  a vertex   $u \in e_1\cap e_2,$ distinct from $v_1$ and $v_2$. If $|e_1|$ and $|e_2|$ is atleast $ s$,  then by the definition of $s$-section, there exists edges $e_1^\prime, e_2^\prime \in E([\bH]_s)$ such that $v_1 \in e_1^\prime \subset e_1, v_2 \in  e_2^\prime \subset e_2$ and $ |e_1^\prime| =|e_2^\prime|=s$, $u \in  e_1^\prime \cap e_2^\prime$. Note that there exists no edge  $e^\prime \in  E([\bH]_s)$ containing both $v_1$ and $v_2$. However, if  $|e_1| < s$ or $|e_2|  < s $, then either $e_1 = e_1^\prime$ or $e_2 = e_2^\prime$ with $v_1 \in e_1^\prime,    v_2  \in e_2^\prime$ and  $u \in  e_1^\prime \cap e_2^\prime$.
Therefore,  $|f(v_1) - f(v_2)| \geq k$, whenever $ v_1 \in e_1,  v_2 \in e_2 \mbox{ and }  u \in e_1\cap e_2.$ Thus, $f$ is well defined $L(h,k)$-colouring of $\bH$ and consequently, $\lambda_{h,k}(\bH) \leq \lambda_{h,k}([\bH]_r)$. To prove $f$ is a minimal $L(h,k)$-colouring of $\bH$. Assume that there exists some other $L(h,k)$-colouring $g$ of $\bH$ such that $span(g) <  \lambda_{h,k}([\bH]_r)$. Since  any  $L(h,k)$-colouring of $\bH$  is an $L(h,k)$-colouring of $[\bH]_r$. Then in particular $g$ is an $L(h,k)$-colouring of $[\bH]_r$  with $span(g) <  \lambda_{h,k}([\bH]_r)$, a contradiction. Therefore the result follows.

\end{proof}

\subsection{ $\lambda$-chromatic number of Hypertrees}
\label{5}
An $r$-uniform hypergraph $\bH$ is called a \textit{star-hypergraph}  if there exists  $\mathcal{C}\subset V$ called the \textit{centre} of the star-hypergraph such that $e_i \cap e_j =\mathcal{C}$ for all $e_i, e_j \in E$. If  $|E|=m$ and $|\mathcal{C}|=c$, then the star-hypergraph is denoted by $\mathbb{K}_{c,m}^r$.

Let $f$ be an $L(h,k)$-colouring of $\bH$. If there are vertices $u, v\in V(\bH)$ such that $f(u)= h'-1$, $f(v)=h'+1$ and no vertex $x\in V(\bH)$ with $f(x)=h'$, then the positive integer $h'$ is said to be a \textit{hole} of $f$.  
 \begin{lem}
\label{lem2} Let $\mathbb{K}_{c,m}^r$ be a star-hypergraph with at least two edges. Then the following hold, $$\lambda(\mathbb{K}_{c,m}^r) = m(r-c)+2c-1.$$
\end{lem}
\begin{proof} 
 Let $\mathcal{C}= \big\{u_s: 1 \leq s \leq c\big\}$ be the centre of the star-hypergraph $\mathbb{K}_{c,m}^r$ and  $e_l = \big\{ v_j^l : 1 \leq j \leq r-c\big\} \cup \mathcal{C}$  be the collection of edges in it, where $1 \leq l \leq m$. Define $L(2,1)$-colouring $f:V(\mathbb{K}_{c,m}^r) \longrightarrow \mathbb{Z}_{\geq 0}$ by, 
 \begin{eqnarray*}
 f(u_s)&=&2s-2  \\
 f(v_j^l)&=&(j-1)m+2c+l-1.\end{eqnarray*}
Therefore, $span(f) =  m(r-c)+2c-1$. Since any two edges $e_1, e_2 \in E(\mathbb{K}_{c,m}^r)$  have a non-empty intersection and $\mathcal{C} \subset e$ for every edge $e \in E(\mathbb{K}_{c,m}^r)$. It follows that $f$ is a well-defined $L(2,1)$-colouring. Therefore,
  $\lambda(\mathbb{K}_{c,m}^r) 
  \leq m(r-c)+2c-1$. The hypergraph $\mathbb{K}_{c,m}^r$ is of diameter two and $\mathcal{C}$  is contained in every edge of $\mathbb{K}_{c,m}^r$. Therefore any $L(2,1)$-colouring of $\mathbb{K}_{c,m}^r$ has atleast $c$ number of holes. This implies, $f$ is minimal and the result follows.
\end{proof}
Note that if $r=2$ and $c=1 $, then $\mathbb{K}_{c,m}^r$ is isomorphic to the star graph $\mathbb{K}_{1, m}$ and $\lambda(\mathbb{K}_{1,m})=m+1$

A connected hypergraph $\bH$ is said to be a \textit{hypertree} if it is linear and removal of any of its  edge $e$ results in a disconnected hypergraph $\bH - e$ and the number of connected components in $\bH - e$ is $|e|$. 
\begin{thm}
\label{hypt}
If $\bH$ is an $r$-uniform hypertree with maximum degree $\Delta$, then
$\lambda(\bH)$ is either $\Delta(r-1)+1$ or $\Delta(r-1)+2.$
\end{thm}
\begin{proof}
Let $v \in V(\bH)$ be a vertex with maximum degree $\Delta$. This implies that the star $H(v)$ rooted at $v$ is isomorphic to $\mathbb{K}_{1,\Delta}^r$, and therefore, $\lambda(\bH) \geq   \lambda(\mathbb{K}_{1,\Delta}^r) =\Delta(r-1)+1$. Colour the vertices of  $H(v)$ as given in Lemma \eqref{lem2}. Consider a vertex $ u\in V(\bH(v))$ such that $V(\bH(u)) \setminus V(\bH(v)) \neq \phi$. Suppose $u$ is contained in $e_u\in E(H(v))$ and $k$ is the colour given to $u$. To colour the remaining vertices of $H(u)$ we need to avoid the colours $k-1$, $k+1$, $k$ and colours given to vertices in $e_u \setminus\{u\}$, that is, we do not use $r+2$ colours from the set $\mathcal{L}=\big\{0, 1, \cdots , \Delta(r-1)+1\big\}$.  Therefore the number of available colours from the set $\mathcal{L}$ to colour the  uncoloured vertices in $H(u)$ is $\Delta(r-1)-r$. Note that the number of uncoloured  vertices in $H(u)$ is at most $(\Delta-1)(r-1).$ Thus we need at most $(\Delta-1)(r-1) -\Delta(r-1)+r = 1$ colour  other than the colours in  the set $\mathcal{L}$ to colour all the uncoloured vertices of $H(u)$. Again, if there is a vertex $w \in \bH(v)$ distinct from $u$ such that $V(\bH(w)) \setminus V(\bH(v)) \neq \phi$, then proceeding by the similar process as above all the vertices of $\bH(w)$ are coloured. Continuing in the same manner we achieve the colouring of $\bH$ and therefore, $\lambda(\bH)$ is either $\Delta(r-1)+1$ or $\Delta(r-1)+2.$
\end{proof} 
The authors in \cite{GY} provide a sharp upper bound on the $\lambda$-chromatic number for trees. Theorem \eqref{hypt} is an interesting generalization to hypertrees. However, it is very difficult to characterize all hypertrees whose $\lambda$-chromatic number is $\Delta(r-1)+1$ or $\Delta(r-1)+2$. Substituting, $c=1$ in Lemma \eqref{lem2}, we get one of the classes of hypertrees for which the bound $\Delta(r-1)+1$ is achieved. Below, we define a \textit{hyperpath} as a generalization of a path in simple graphs and compute its $\lambda$-chromatic number. This provides us a class of hypertrees  for which another bound $\Delta(r-1)+2$ is achieved.
\begin{defn}
A $r$-uniform  hypergraph with $m$  number of edges is said to be a hyperpath denoted by \enquote{~$ \mathbb{P}^m_r$~}  if there exists sequence of edges and vertices $e_1v_1e_2v_2\cdots e_{m-1}v_{m-1}e_m$ such that for $1\leq i < j \leq m$,
\begin{equation*}
e_i\cap e_j= 
\begin{cases}
 \phi & j \neq i+1 \\
 v_i & j = i+1. 
\end{cases}
\end{equation*} 
\end{defn}
\begin{thm}
\label{hp}
Let $ \mathbb{P}^m_r$ be a $r$-uniform  hyperpath, where $r\geq 3$. Then the following hold,
\begin{equation*}
 \lambda(\mathbb{P}^m_r) =
\begin{cases}
 2r-2 & \text{ if }m=1 \\
  2r-1 & \text{ if } m= 2\\
  2r  &  \text{ if }m\geq 3.
\end{cases}
\end{equation*} 

\end{thm}
\begin{proof}
The cases $m=1,2$ follow trivially.  Let $\mathbb{P}^m_r$ be a hyperpath on $m$ edges, $m \geq 3$. We prove the result by induction on $m$. For $m=3$, let $\mathbb{P}^3_r =e_1v_1e_2v_2e_3$ be a $r$-uniform hyperpath on $3$ edges, where $V(\mathbb{P}^3_r) =\{ v_{i_j}^j, v_1, v_2\}$, $ e_j = \{ v_{i_j}^j\} \cup \{v_j\}, 1\leq j \leq 3 $. Define $L(2,1)$-colouring $f:\longrightarrow \mathbb{Z}_{\geq 0}$ by,\\
\begin{eqnarray*}
f(v_1)&=& 0,\\
f(v_{i_1}^1)&=&  ~2i_1+1,1\leq i_1\leq r-1,\\
f(v_{i_2}^2)&=& 2i_2, ~1\leq i_2\leq r-2,\\
f(v_{i_3}^3)&=& 2i_3-1, ~1\leq i_3 \leq r-1,\\
f(v_2)&=& 2r.
\end{eqnarray*}
This completes the $L(2,1)$-colouring of all vertices of $\mathbb{P}^3_r$. $f$ is  well-defined, since every pair of adjacent vertices have colours that differ by at least $2$ and any pair of  vertices whose corresponding edges have a non-empty intersection has colours that differ by at least 1. Therefore, $\lambda(\mathbb{P}^3_r) \leq 2r$. 

The star $\bH(v_1)$ is isomorphic to $\mathbb{K}^r_{1, 2}$, therefore by Lemma \eqref{lem2}, $\mathbb{P}^{3}_r \geq 2r - 1$. To colour the remaining vertices of $\mathbb{P}^{3}_r$, we need to colour the vertices of star $\bH(v_2)$, where $\bH(v_1) \cap \bH(v_2) = e$. We follow the same argument as in Theorem \eqref{hypt} and conclude that $\lambda(\mathbb{P}^{3}_r) = 2r$. For $m=k+1$, let $\mathbb{P}^{k+1}_r = e_1v_1e_2v_2\cdots v_ke_{k+1}$ be a $r$-uniform hyperpath on $k+1$ edges. Removing the edge $e_{k+1}$ from $\mathbb{P}^{k+1}_r$, we are left with the hyperpath $\mathbb{P}^{k}_r = e_1v_1e_2v_2\cdots v_{k-1}e_k$. By induction hypothesis, $\lambda(\mathbb{P}^k_r) = 2r$ and this minimal colouring assigns colour to the vertices in edges $e_1, e_2, \cdots, e_k$ of   $\mathbb{P}^{k+1}_r$. The remaining vertices in $e_{k+1}$ are coloured by the colours given to vertices of $e_{k-1}$. It is a well defined  minimal $L(2,1)$-colouring of $\mathbb{P}^{k+1}_r$, since $e_{k-1} \cap e_{k+1} = \phi $.
\end{proof}

 \section{$\lambda$-Chromatic number of Cartesian product} 
 \label{s4}
There are various surprising variety of hypergraph products that have been discussed earlier. Most of the hypergraph products are generalisations of four \enquote{standard graph products} called as, \textit{Cartesian product, Tensor product, Strong product} and \textit{Lexicographic product}. These graph products behave reasonably well in an algebraic way, while maintaining the salient structure of their factors. In this section, we focus on the Cartesian product of finite  hypergraphs with finitely  many factors.
 \begin{defn}[{\bf Cartesian Product}] Let $\{ \bH_i~ |~ i \in I\}$ be a family of hypergraphs. The Cartesian product $\bH = \tiny{ \square} ~\bH_i$ is a hypergraph with vertex set $V(\bH) = \times_{i \in I} V(\bH_i)$ and edge set $E(\bH) = \{ E \subset V(\bH)~ |~ p_i(E) \in E(\bH_i) ~ for~ exactly~ one ~i \in I,~ and ~|p_j(E)| =1 ~for~ j \neq i \}$. For $i \in I$,  $p_i:V(\bH) \longrightarrow V(\bH_i)$ is the projection of the cartesian product on the $i^{th}$ coordinate. 
\end{defn}
The Cartesian graph product of simple graphs was introduced by Gert Sabidussi \cite{GS}, who showed that simple connected graphs have a unique Cartesian prime factor decomposition. This result was generalized by Wilfried Imrich \cite{WI} to simple hypergraphs. It is widely known that many significant graph invariants including $\lambda$-chromatic number propagate under graph products.  In fact, for all  values of $m$ and $n$, the exact values of the $\lambda$-chromatic number of the product of two paths $P_m$ and $P_n$ is determined in \cite{WGM}. In \cite{GMS}, the authors compute the  exact value of the $\lambda$-chromatic number of the product of two complete graphs $K_m$ and $K_n$ and of the graph $K_p \tiny{\square} \cdots  \tiny{\square} K_p$, where $p$ is a prime number. 

We arrange the vertices of the Cartesian product of two hypergraphs in a rectangular array. This allow us to colour the vertices conveniently. To illustrate the concept, we arrange the vertices of the cartesian product of two complete hypergraphs  $\bH_{10}$ and $\bH_{13}$ of order $10$ and $13$ respectively and determine its $\lambda$-chromatic number. 
\begin{figure}[h]

\centering
 \includegraphics[scale=.39]{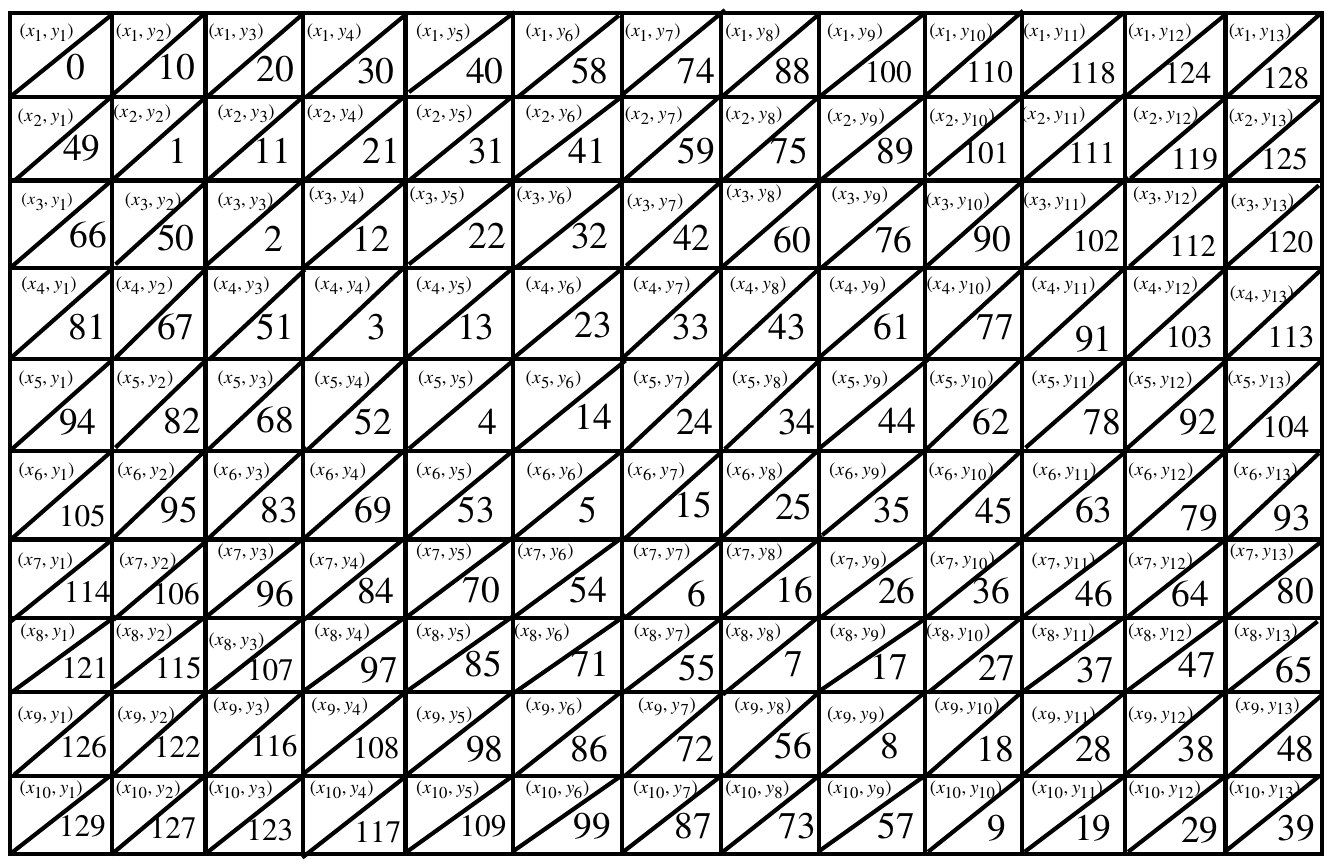}
     \caption{Minimum $L(2,1)$-colouring of $\bH_{10} \square ~ \bH_{13}$}

 \end{figure}
 
From the Figure (2), we see that $\lambda(\bH_{10} \square ~ \bH_{13}) = 129$.
In general, given two complete hypergraphs $\bH_m$ and $\bH_n$ of order m and n respectively, $\lambda(\bH_m \square ~ \bH_n)= mn-1$ with the minimum colouring provided below, where $V(\bH_n)=\{x_i~|~ 1 \leq i \leq n\}$, $V(\bH_m)=\{y_j~|~ 1 \leq j \leq m\}$ and $n \leq m$.
     \begin{equation*}
 f(x_i,y_j)= 
\begin{cases}
i-1 & \text{ if } i=j \\
 ln+i-1 & \text{ if } j = i+l, \text{ where } 1 \leq l \leq m-n+1\\
 mn+2n-n^2+i-2 &   \text{ if  } j= i-1\\
 n(m+2k)-n^2-k^2-1+i  & \text{ if } j = i-k,  \text{ where } 2 \leq k \leq n-1\\
 n(m+1)+k(2n-k-1)-n^2-3+i &  \text{ if } j = m-n+k+i,  \text{ where } 1 \leq k \leq n-1.
\end{cases}
\end{equation*} 

     The colouring is minimal, since  the graph has diameter 2 and the map is one-one. 
     
We arrange the vertices of the Cartesian product $\bH~ \square~ \bH^\prime$ of two hypergraphs $\bH$ and $\bH^\prime$ in the above-described rectangular array, where $x_i \in V(\bH)$ and $y_j \in  V(\bH^\prime)$. To investigate its $L(2,1)$-colouring, the following observations can be made about the rectangular array.
 \begin{itemize}
 \item Any row or column in the array induces a copy of  $\bH^\prime$ or $\bH$ respectively.
 \item The distance between two vertices in separate rows and columns is at least two.
 \item $d((x_i,y_j),(x_{i^\prime},y_{j^\prime}))= d(x_i, x_{i^\prime})+ d(y_j,y_j{^\prime})$
 \item If $M \subset V(\bH)$  consists of vertices whose mutual distance is $k$ and $N\subset V(\bH^\prime)$  consists of vertices whose mutual distance is $l$, then $M \times N \subset V(\bH ~\square ~\bH^\prime)$  consists of  vertices whose mutual distance is $k+l$. If $M$ and $N$ are maximal in the sense that they contain the maximum number of vertices with the given property, then so is   $M \times N$.
 \end{itemize}
In the following result, we determine the $\lambda$-chromatic number of  Cartesian product  of  star-hypergraph $\mathbb{K}_{c,m}^r$ and a complete hypergraph $\mathbb{H}_n$.

\begin{thm}
\label{abc}
let $\mathbb{K}_{c,m}^r$ be a $r$-uniform star-hypergraph and $\bH_n$ a complete hypergraph. For $m <n $, the following hold,
$$\lambda(  \mathbb{K}_{c,m}^r \tiny{\square}~ \bH_n)= nr-1. $$
\end{thm}
\begin{proof}
Let $\mathcal{C} = \{v_1, v_2, \cdots v_c\}$ be the central vertices of $\mathbb{K}_{c,m}^r$ and let $\mathcal{C^\prime}=\{y_1, y_2, \cdots y_{m(r-c)}\}$ be its non-central vertices. Consider its edges as $e_l =\{ y_l, y_{l+m}, y_{l+2m}\cdots \} \cup \mathcal{C}$, where $1\leq l \leq m$. Moreover, let $V(\bH_n)=\{x_1, x_2, \cdots x_n \}$.  The vertices of $\bH =  \mathbb{K}_{c,m}^r \tiny{\square}~ \bH_n$ are arranged in the following $(m(r-c)+c) \times n$ rectangular array given in Figure (3).

\begin{figure}[h]
 \label{Fig:1}
\centering
 \includegraphics[scale=.39]{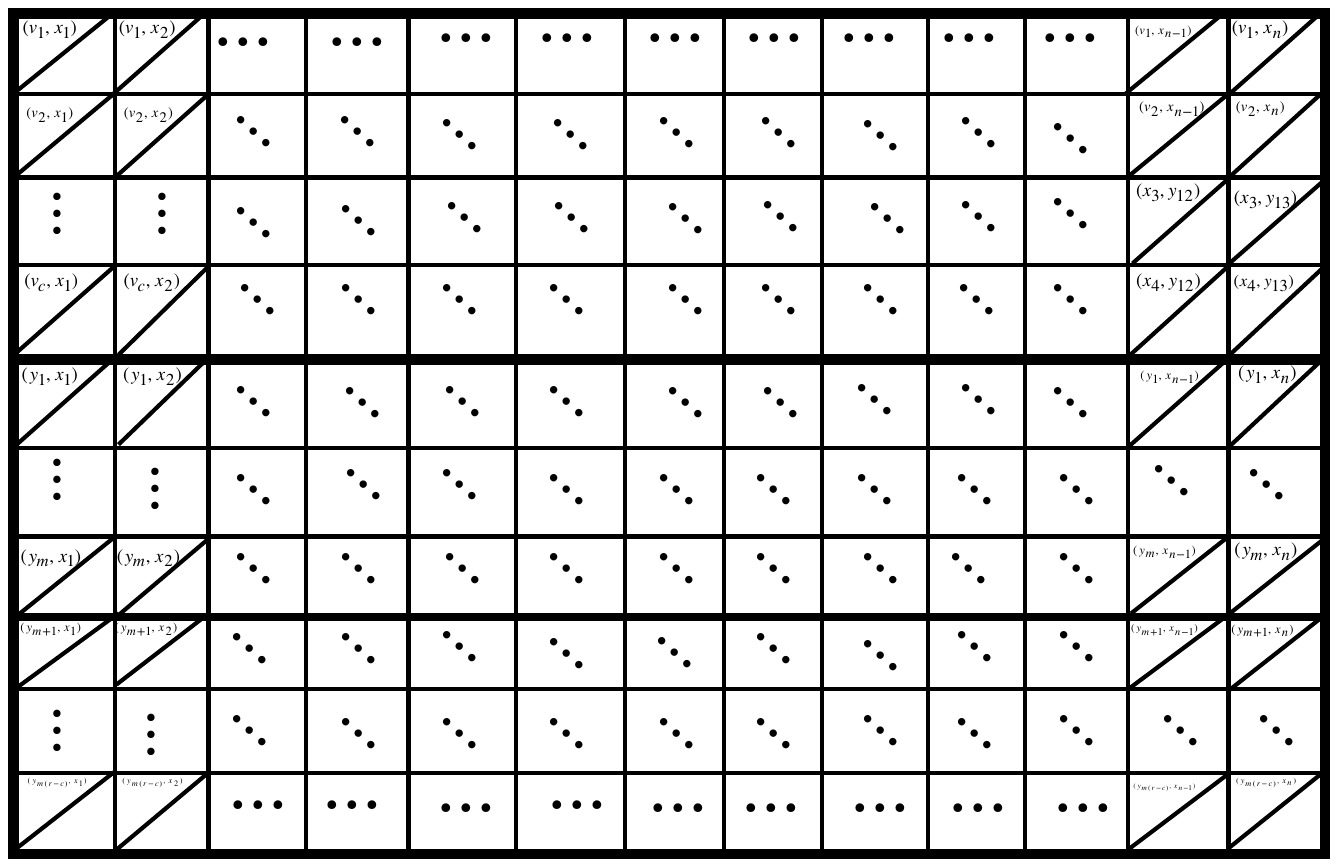}
     \caption{ $\mathbb{K}_{c,m}^r \square ~ \bH_{n}$}

 \end{figure}
Every edge in $\bH$  resides  completely either in a row or in a column of the array. Therefore,  $d\big((v_j,x_i),(v_{j^\prime} , x_{i^\prime})\big)= 2$,  $1 \leq i < i^\prime  \leq n $ and $1 \leq j < j^\prime \leq c.$ Consider the set $V_k=\{ y_{k+s}~|~0\leq s \leq m-1\} \subset \mathcal{C^\prime}$, where $k= 1 , m+1, 2m+1 \cdots $ and note that $d(x,y) =2$ for distinct $x,y \in  V_k$. Therefore for each $k$, $x \neq y \in V_k$, $d((x,x_i),(y,x_{i^\prime })) =3$, where $x_i \neq x_{i^\prime } \in V(\bH_n)$.  Partition the above-mentioned rectangular array into one $c \times n$ rectangular block  and $(r-c)$ number of $m \times n$ rectangular blocks as illustrated in the Figure (3). Observe that the initial $c \times n$ block induces  the Cartesian product of two complete hypergraphs $\bH_c$ and $\bH_n$. Furthermore, partition  each of the remaining $(r-c)$ blocks into $n$ sets $U^h_i$, where $1 \leq h \leq (r-c)$. The cardinality of each set of the type $U_i^h$ is $m$ and is, either a diagonal or a union of two diagonals in the block with no two vertices in a set sharing a column or a row. Again proceeding further and partition the first block into $n\ceil{\frac{n}{c}}$ diagonals say $D_t$ with lengths not more than $c$. For each $i$, $1\leq i \leq  n$, it is clear that any two vertices in $U^h_i$ do not share edges with non-empty intersection whereas for each $t$, any two vertices in $D_t$ share edges with non-empty intersection.

To color the vertices of $\bH$, start from $(v_1, x_1)$ and move downward through the diagonal $D_1$, make use of the least possible  available colours $\{ 0, 1 , \cdots c-1\}$. Similarly, colour the vertices through the diagonal $D_2$ starting from the $(c+1)^{th}$ row of the rectangular array (if any) and use the least available colours. Continuing in the same manner until we reach a vertex say $(v, x) \in D_{t^\prime}$, $1 \leq t^\prime < n\ceil{\frac{n}{c}}$ in the last row of the first block. Next, pick a set from the collection $\{U_i^1\}_{i=1}^{n}$ say $U_{i^\prime}^1$ whose first vertex  does not share a  column with $(v,x)$ and colour all its vertices with  a colour which exceeds one than the maximum colour in $D_{t^\prime}$. Moving to next block, again pick a set from the collection $\{U_i^2\}_{i=1}^{n}$ say $U_{i^{\prime \prime}}^2$, and colour its vertices by a colour that exceeds one more than that of the colours  given to vertices in $U_{i^\prime}^1$. Following the same process of colouring till we reach the last row of the array. Repeating  the same procedure from the first block until all the vertices of the hypergraph are coloured. Note that for each $i$ and $h$ we utilise a single colour to colour the vertices of $U_i^h$, that is each set of the type $U_i^h$ is monochromatic. The diagonals $D_t$ are coloured by $|D_t|$ number of colours. Consequently for each $m \times n$ rectangular block, we utilise $n$ colours to colour its vertices and $nc$ colours to colour the vertices of first $n\times c$ block.  It is a well defined $L(2,1)$-colouring, since we begin the colouring from zero. So, $\lambda( \mathbb{K}_{c,m}^r \tiny{\square}~ \bH_n) \leq nr-1.$ 

The colouring is minimal, since none of the colours in the lower blocks can be repeated more than $m$ times and we cannot use the same colour for vertices in the lower $m \times n$ rectangle blocks as the colour used in the first block. This is due to the fact that any vertex in a $m \times n$ rectangular blocks lying in edge $e$ has non-empty intersection with edge $e^\prime$ in the $c \times n$ rectangular block, and any vertex whose colour is repeated  more than $m$  shares an  edge  with  atleast one of the other $m$ vertices. Therefore the result follows.
  \end{proof}
Figure (4) below demonstrates the approach used in Theorem \eqref{abc} and we compute the  $\lambda$-chromatic number of $\mathbb{K}_{4,3}^7 \square ~ \bH_{14}$ as $97$.
 
  \begin{figure}[h]
  \label{fig4}
\centering
 \includegraphics[scale=.39]{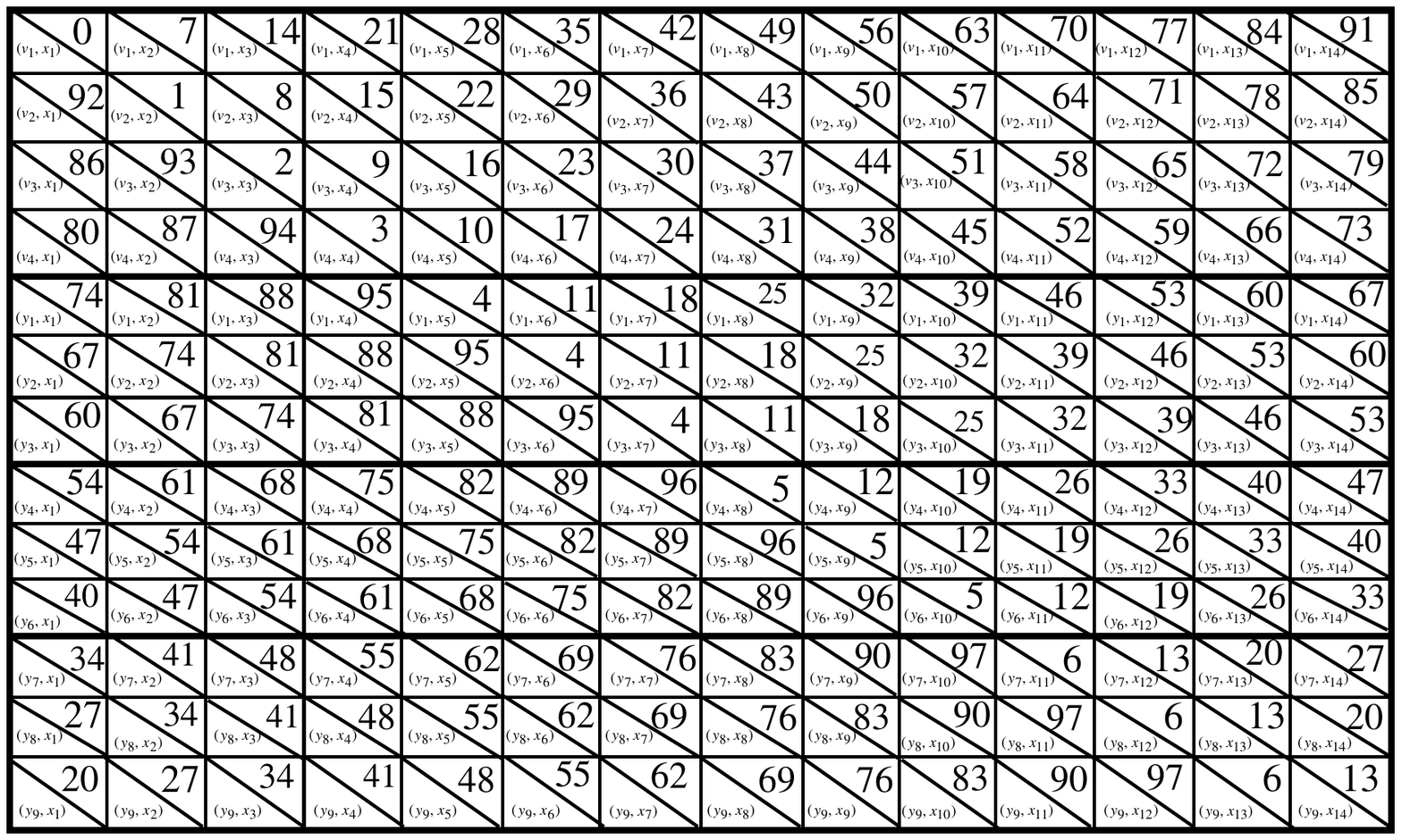}
     \caption{Minimum  $\lambda$-colouring of  $\mathbb{K}_{4,3}^7 \square ~ \bH_{14}$.}
 \label{Fig:1}
 \end{figure}

Finally, we determine the $\lambda$-chromatic number and the strong chromatic number of the product  $\bH_n \tiny{\square}~ \mathbb{K}_{c,m}^r$.
  \begin{thm}
 Let $\mathbb{K}_{c,m}^r$ be a star-hypergraph and $\bH_n$ a complete hypergraph. Then the following hold,
  \begin{equation*}
\lambda_{h,k}( \mathbb{K}_{c,m}^r \square~ \bH_n)=
\begin{cases}
  knr-k  & \text{ if }kr \geq h \\
  k(r-1)+h(n-1) & \text{ if }kr<h
\end{cases}
\end{equation*} 
and 
$$ \chi^\prime(\bH_n \square~ \mathbb{K}_{c,m}^r) = \min\{n,c\}+r-c,$$
where $h, k$ are any two positive integers with $h > k$.
  \end{thm}
  
\begin{proof}
The proofs follows by the same argument as in Theorem \eqref{abc}.
\end{proof}
  
{\bf Conclusion.}\\
In this research article, we discussed the generalisation of $L(h,k)$-colourings of graphs to hypergraphs and  explored the relationship between the $\lambda$-chromatic number of a connected $k$-regular graph to its expansion constant and spectral gap. We studied the relationship between the  spectral gap  of $k$-regular, $r$-uniform hypergraph $\bH$ and  its $\lambda$-chromatic number. We provided bounds for $L(h,k)$-chromatic number of a hypergraph  in terms of different hypergraph invariants. We established a relationship which connects the chromatic number of  the line graph associated with a hypergraph to the $\lambda$-chromatic number of the hypergraph. We determined the sharp bounds for $\lambda$-chromatic number of hypertrees. Finally, we discussed the distribution of $\lambda$-chromatic number over cartesian product of some classes of hypergraphs.

{\bf Acknowledgement:} The first author's research is partially supported by the University Grants Commission, Govt. of India under UGC-Ref. No. 191620023547.

\textbf{Data Availability}\\
Data sharing not applicable to this article as no datasets were generated or analysed during the current study.

\textbf{Declarations}

\textbf{Conflict of interest} The authors declare that there is no conflict of interest.


\begin{thebibliography}{AAAA}
\bibliographystyle{alpha} 
\bibitem{AM} Alon, N., and Milman, V. D.  :  $\lambda_1$, isoperimetric inequalities for graphs, and superconcentrators. J. Comb. Theory, Ser. B, {\bf38(1)}, 73-88 (1985).
\bibitem{CK} Chang, G. J., Kuo, D. : The $L(2,1)$-labeling problem on graphs. SIAM J. Disc. Math., {\bf9}, 309 – 316 (1996).
\bibitem{CJ} Chang, G. J., Ke, W.-T., Kuo, D., Liu, D. D.-F. and Yeh, R. K. :  On $L(d,1)$-labelings of graphs. Disc. Math., {\bf220}, 57–66 (2000). 
\bibitem{D} Dodziuk, J. : Difference equations, isoperimetric inequality and transience of certain random walks. Trans.  Amer. Math. Soc., {\bf284(2)}, 787-794 (1984). 
\bibitem{GMS} Georges, J. P., Mauro, D. W., and Stein, M. I. : Labeling products of complete graphs with a condition at distance two. SIAM J.  Disc. Math., {\bf14(1)}, 28-35  (2001).
\bibitem{G} Gonçalves, D. : On the $L(p,1)$-colouring of graphs. Disc. Math., {\bf308(8)}  : 1405 - 1414 (2008).
\bibitem{GY} Griggs, J.,  Yeh, R. : colouring graphs with a condition at distance two. SIAM J. Disc. Math., {\bf5},  586-595 (1992) .
\bibitem{WI} Imrich, W. :  über das schwache Kartesische Produkt von Graphen. J.  Comb. Theory, Ser. B, {\bf11(1)}, 1-16 (1971).
\bibitem{R} Roberts,  F.S. : T-colorings of graphs: recent results and open problems. Disc. Math., {\bf93},  229-245 (1991).
\bibitem{GS} Sabidussi, G. : Graph multiplication.  Mathematische Zeitschrift,  {\bf72(1)},  446-457 (1959).
\bibitem{WGM} Whittlesey, M. A., Georges, J. P., and Mauro, D. W. : On the $\lambda$-number of $Q_n$ and related graphs. SIAM J.  Disc. Math., {\bf8(4)}, 499-506 (1995).



\end{thebibliography}
\end{document}